%% file: quotient-maps.tex
\documentclass[a4paper,11pt,reqno]{amsart}

\input{header-quotient}

\title{Maps of Mori dream spaces}

\author{Andreas Hochenegger}
\address{Dipartimento di Matematica ``Federigo Enriques'', Universit\`a degli Studi di Milano, via Cesare Saldini 50, 20133 Milano, Italy}
\email{andreas.hochenegger@unimi.it}
\author{Elena Martinengo}
\address{Dipartimento di Matematica ``Giuseppe Peano'', Universit\`a degli Studi di Torino, via Carlo Alberto 10, 10123 Torino, Italy}
\email{elena.martinengo@unito.it}

\keywords{Mori dream spaces, Mori dream stacks, maps in homogeneous coordinates}
\subjclass[2010]{14A20, 14L24, 14E30}

\begin{document}

\maketitle

\begin{abstract}
Let $\phi\colon X \to Y$ be a map of $\QQ$-factorial Mori dream spaces. 
We prove that there is a unique \emph{Cox lift} $\Phi\colon \XX \to \YY$ of Mori dream stacks coming from a homogeneous homomorphism $\RR(Y) = \RR(\YY) \to \RR(\XX)$, where $\YY$ is a canonical stack of $Y$ and $\XX$ is obtained from $X$ by root constructions, and $\phi$ is induced from $\Phi$ by passing to coarse moduli spaces.
We also apply this techniques to show that a Mori dream quotient stack is obtained by roots from its canonical stack.
\end{abstract}


\section{Introduction}
Given a map $\phi\colon X \to Y$ of $\QQ$-factorial Mori dream spaces, 
one can ask whether this map lifts to a homogeneous homomorphism $\RR(Y) \to \RR(X)$ of Cox rings. 
As soon as $Y$ is singular, such a homomorphism needs not to exist, as pulling back Weil divisors is not well-defined.

In \cite{Brown-Buczynski}, Gavin Brown and Jaros\l aw Buczy\'nski show how to lift rational maps between toric varieties to (multi-valued) maps between their respective Cox rings.
In this article, we show that this construction has a geometrical interpretation involving quotient stacks and works well for Mori dream spaces. 
This leads to our main result; see \autoref{thm:lift} for a more precise and more general version.

\begin{maintheorem}
\hypertarget{maintheorem}{}
Let $\phi\colon X \to Y$ be a map between $\QQ$-factorial Mori dream spaces 
where $X$ is complete with factorial Cox ring $\RR(X)$.
There is a Mori dream quotient stack $\XX$ with coarse moduli space $X$ with the
following properties:
\begin{enumerate}
\setlength{\multicolsep}{0ex}
\item  \label{itm:mainlift_root} The stack $\XX$ is built from the canonical Mori dream stack $\Xass$ of $X$ by root constructions
with prime divisors and line bundles. 
\item  \label{itm:mainlift_quotient} There is a homogeneous morphism $\Phi^*\colon
\RR(Y) \to \RR(\XX)$ such that
the following diagrams commute: 
\[
\xymatrix@R=3ex@C=1em{
\RR(Y) \ar[r]^{\Phi^*} & \RR(\XX) &&
 \XX \ar[r]^-\Phi \ar[d] & \YY^{\ass} \ar[d]\\
\RR(Y; \Pic(Y)) \ar[u] \ar[r]_-{\phi^*} & \RR(X) \ar[u] &&
 X \ar[r]_\phi & Y
}
\]
where $\YY^{\ass}$ is the canonical Mori dream stack of $Y$. 
\item  The stack $\XX$ is minimal:
given any Mori dream quotient stack $\XX'$ that satisfies \eqref{itm:mainlift_quotient},
then the map $\XX'\to X$ to the coarse moduli space factors through $\XX$.
\end{enumerate}
We call $\Phi\colon\XX \to \YY^{\ass}$ the \emph{Cox lift} of the map $\phi\colon X \to Y$.
\end{maintheorem}

The following example illustrates this theorem.

\begin{example}[{cf.\ \cite[Ex.~1.1.1]{Brown-Buczynski}}]
\label{ex:A2_mod_mu2}
Let $\xi$ be a primitive $k$-th root of unity acting on $\AA^2$ by $\xi \cdot (x,y) = (\xi^a x, \xi^b y)$. The quotient $\AA^2/\mu_k$ is then called a $\frac1k(a,b)$-singularity.
Consider the $\frac12(1,1)$-singularity $\AA^2/\mu_2$.
In the coordinates $\AA^2 = \Spec \kk[x,y]$, the quotient is given as
\[
\AA^2/\mu_2 = \Spec \kk[x,y]^{\ZZ_2} = \Spec \kk[x^2,xy,y^2] = \Spec \kk[u,v,w]/uw-v^2.
\]
Geometrically, this is a double cone with an isolated singularity at $0$.
Let $\AA^1 = \Spec \kk[t]$ be a Weil divisor passing through the singular origin. For example, we can choose the embedding $\phi\colon\AA^1 \into \AA^2/\mu_2$ as in \autoref{fig:A2_mod_mu2} given in terms of the coordinate rings as
\[
\FctArray{
\phi^*\colon & \kk[x^2,xy,y^2] & \to & \kk[t] \\
& x^2     & \mapsto & t\\
& xy, y^2 & \mapsto & 0\\
}
\]

The Cox ring of $Y = \AA^2/\mu_2$ is $\RR(Y) = \kk[x,y]$ graded by $\Cl(Y) = \ZZ_2$, which is given by $\deg(x)=\deg(y) = 1$. The Cox ring of $X = \AA^1$ is again the ring $\RR(X) = \kk[t]$ which is graded trivially since $\Cl(X)=0$.
As the authors of \cite{Brown-Buczynski} point out, the best way to lift the map $\phi^*$ to the Cox rings is
\[
\FctArray{
\Phi^*\colon & \kk[x,y] & \to & \kk[\sqrt{t}] \\
& x & \mapsto & \sqrt{t}\\
& y & \mapsto & 0
}
\]
Here $\sqrt{t}$ should mean that we consider an extension of $\kk[t]$, namely, the ring $\kk[t][z]/(z^2-t) = \kk[\sqrt{t}]$, where we write suggestively $\sqrt{t} \coloneqq z$.
As shown in \cite{Brown-Buczynski}, this map in homogeneous coordinates has several nice properties, like the images of points or the pullback of divisors can be easily computed.

Our main observation is that the map of Cox rings above is inherently a map between toric stacks.
The singular space $\AA^2/\mu_2$ is the coarse moduli space of the smooth quotient stack
\[
\YY = \stackquot{\Spec \kk[x,y]}{\mu_2}.
\]
Although the origin is a smooth point of $\YY$, it is still different from all other points, since it has the non-trivial stabiliser $\mu_2$.

So if we consider again the divisor as above through the origin, the corresponding point on the embedded divisor should also have the $\mu_2$-stabiliser.
We can obtain such a stack by performing a root construction at the origin on $X = \AA^1$, to get $\XX = \sqrt[2]{0/\AA^1}$. This stack has the following description as a quotient
\[
\XX = \stackquot{\Spec \kk[\sqrt{t}]}{\mu_2},
\]
where $\mu_2$ acts by swapping $\sqrt{t} \leftrightarrow -\sqrt{t}$.

The inclusion $\Phi\colon\XX \into \YY$ induces naturally the map $\Phi^*$ above, which becomes just the pullback of global sections.
Moreover, it induces the map $\phi\colon\AA^1 \into \AA^2/\mu_2$ from the beginning (by passing to the coarse moduli spaces).
\end{example}

\begin{figure}
\begin{tikzpicture}[scale=0.4,thick]
\draw (-2,-5) -- (0,0);
\draw [fill=white] (-1.99,3.98) -- (1.99,3.98) -- (0,0) -- cycle;
\draw [fill=gray!20] (0,4) circle (1.99cm and 0.4cm);
\draw [fill=white] (0,-4) circle (1.99cm and 0.4cm);
\draw [fill=white] (-1.99,-3.98) --(0,0) -- (1.99,-3.98);
\draw (0,0) -- (2,5);
\draw [fill=gray] (0,0) circle (0.1cm);
\end{tikzpicture}
\hspace{2cm}
\begin{tikzpicture}[scale=0.4,thick]
\draw (-4.5,-4.5) -- (-4.5,4.5) -- (4.5,4.5) -- (4.5,-4.5) -- (-4.5,-4.5);
\draw (-2,-5) -- (2,5);
\draw (.5,-.2) node {\rotatebox{50}{$\circlearrowleft$}};
\draw (1.3,-.2) node {$\scriptstyle\mu_2$};
\draw [fill=gray!20] (0,0) circle (0.1cm);
\end{tikzpicture}
\caption{The embedding of $\AA^1$ into $\AA^2/\mu_2$ and of $\XX$ into $\YY$.}
\label{fig:A2_mod_mu2}
\end{figure}

The key idea of this article is that, given a map $\phi\colon X \to Y$, we may not be able to pull back an $r \in \RR(Y)$ if it comes as a section of a Weil divisor on $Y$. But, when $Y$ is at least $\QQ$-factorial, then some positive multiple $r^n$ becomes Cartier, so there is a pullback $\phi^*(r^n) \in \RR(X)$.
In order to get a homomorphism, we have to add roots of (the factors of) $r^n$ to $\RR(X)$.

It seems that Takeshi Kajiwara observed this for the first time in an example in \cite[Ex.~4.12]{Kajiwara}.
In \cite{Berchtold}, Florian Berchtold addressed also the question of lifting maps to Cox rings, but focusing on the cases where no roots are needed.
The question of lifting rational maps between toric varieties was systematically treated in \cite{Brown-Buczynski} by Gavin Brown and Jaros\l{}aw Buczy\'nski, where roots are used. As they do not use the language of stacks, this leads to the notion of multi-valued maps.

In this article, we partly generalise the results from \cite{Brown-Buczynski} in the sense, that we discuss (regular) maps between $\QQ$-factorial Mori dream spaces instead of rational maps between arbitrary toric varieties.
Additionally, the Cox lift obtained here fulfills a certain uniqueness and can be determined quite algorithmically.
Jaros\l{}aw Buczy\'nski and Oskar K\k edzierski have also obtained independently a generalisation for rational maps between (even not $\QQ$-factorial) Mori dream spaces in \cite{Buczynski-Kedzierski}.
Finally, our construction is similar in spirit to the process of \emph{destackification} as introduced in \cite{Bergh} by Daniel Bergh.

\medskip

This article is organised as follows.
In \autoref{sec:preliminaries} we recall the necessary background. This includes Cox rings, Mori dream spaces, their generalisation to stacks and root constructions.
For the latter, we will refer mostly to our article \cite{HM}. But there is a slight difference; in [loc.~cit.] Mori dreams stacks are required to be smooth, whereas here we only ask for their Cox rings to be graded factorial and normal.
In \autoref{sec:coxlift}, we will prove our main result, \autoref{thm:lift}.
The proof of this theorem is used in \autoref{sec:application} to show that a Mori dream stack is a Mori dream quotient stack if and only if it is obtained by root constructions along prime divisors and line bundles from its canonical stack. This generalises the results on Mori dream quotient stacks of our previous article, where we asked for smoothness.

\subsubsection*{Notations and Conventions}

In this article, $\kk$ will always be an algebraically closed field of characteristic $0$, since most results on Mori dream spaces and GIT quotients are only available in this setting. 
When speaking of algebraic stacks, we work in the \'etale topology.

\subsubsection*{Acknowledgements}
We thank Jaros\l{}aw Buczy\'nski and Oskar K\k edzierski very much for reading carefully through a preliminary version of this article and giving many valuable comments.
Many thanks go also to Antonio Laface for some helpful discussions and explanations.
We are grateful to Rita Pardini and Fabio Tonini for answering our questions.
Finally, we thank anonymous referees for comments improving the article.

\section{Preliminaries}
\label{sec:preliminaries}

In the following, we recall basic facts about Mori dream spaces and stacks, mostly to fix notation. The references are \cite{Hu-Keel,Coxrings} for statements about Mori dream spaces and \cite{HM} for those about Mori dream stacks.
As a general reference for the theory of stacks serves the book \cite{Olsson}. As we are dealing here mainly with quotient stacks, we recommend the article \cite{Alper-Easton} as it gives reformulations of equivariant geometry into the language of quotient stacks.

\subsection*{Mori dream spaces}

We call an irreducible and normal variety $X$ a \emph{Mori dream space} 
if
\begin{enumerate}
\item $X$ is $\QQ$-factorial,
\item $H^0(X,\OO_X^*) = \kk^*$,
\item $\Cl(X)$ is finitely generated as a $\ZZ$-module,
\item its Cox ring $\RR(X)$ is finitely generated as a $\kk$-algebra.
\end{enumerate}
We remark that this is different from the characterisation of Mori dream spaces given in \cite[Prop.~2.9]{Hu-Keel} (where $X$ is assumed to be projective and the quotient is taken by a torus) as the focus there is on the minimal model program. It is almost the same as in \cite[\S 1.6.3]{Coxrings} where the focus is (as here) on quotients (and also certain prevarieties are allowed).

\begin{proposition}[{cf.\ \cite[Prop.~2.9]{Hu-Keel} and \cite[\S 1.6.3]{Coxrings}}]
If $X$ is a Mori dream space, then 
\[
X = \quot{\Spec \RR(X) \setminus V}{\Hom(\Cl(X),\kk^*)}
\]
where the \emph{irrelevant locus} $V = V_X$ is the vanishing locus of the \emph{irrelevant ideal}
\[
\Jirr(X) = ( f \in \RR(X) \text{ homogeneous } \mid X_f \coloneqq X \setminus \{f=0\} \text{ is affine}).
\]
\end{proposition}

The Cox ring appearing in the definition can be naively defined as the $\Cl(X)$-graded vector space
\[
\RR(X) = \bigoplus_{[D] \in \Cl(X)} H^0(X,\OO(D)),
\]
but as pointed out in \cite{Hu-Keel}, this definition does not admit a canonical ring structure. In the case when $\Cl(X)$ is free, after fixing a basis of $\Cl(X)$, the product can be defined by multiplying the sections as rational functions.

For an arbitrary finitely generated $\Cl(X)$, the definition of a ring structure is a bit more technical; we follow \cite[\S I.4]{Coxrings}. 
 
Choose a finite number of divisors $D_1,\ldots,D_m$, such that the classes $[D_1],\ldots,[D_m]$ generate $\Cl(X)$. Denote by $K$ the free subgroup of the group of Weil divisors  
$\WDiv(X)$ generated by these divisors. 
We hence have a surjection $K \onto \Cl(X)$. 
If we write $K^0$ for its kernel, then we obtain a short exact sequence  $K^0 \into K \onto \Cl(X)$ with $K,K^0$ free abelian groups. 
Finally, since the composition of $K^0 \into K \into \WDiv(X)$ and $\WDiv(X) \onto \Cl(X)$ is zero, the map $K^0 \into \WDiv(X)$ has to factor uniquely through $\kk(X)^*/\kk^*$, that is, the kernel of $\WDiv(X)\onto \Cl(X)$. So we obtain a map $K^0 \into \kk(X)^*/\kk^*$. As $K^0$ is a free abelian group, this map lifts (non-uniquely) to $\chi\colon K^0 \into \kk(X)^*$. In particular, for a principal divisor $E \in K^0$ we have $E = \div \chi(E)$. 
 
Consider the sheaf of $\Cl(X)$-graded rings $\mathcal{S} = \bigoplus_{D \in K} \OO_X(D)$ containing the homogeneous ideal $\mathcal{I} = \genby{1 - \chi(E) \mid E \in K^0}$. 
 
\begin{definition} 
The \emph{Cox ring} of $X$ is the $\Cl(X)$-graded ring 
\[ 
\RR(X) = H^0(X,\mathcal{S}) / H^0(X,\mathcal{I}). 
\]  
\end{definition}

The definition of the Cox ring depends on the choices we made above. But for different choices, the resulting rings will be (non-canonically) isomorphic.

We recall some terminology taken from \cite[\S 1.5.3]{Coxrings}.

\begin{definition}
Let $A$ be a finitely generated abelian group and $R$ be a ring graded by $A$. A non-zero, non-invertible homogeneous element $a\in A$ is called \emph{$A$-irreducible}, if $a=bc$ with $b,c \in A$ homogeneous implies that either $b$ or $c$ is invertible.

If the group $A$ is fixed, a $A$-irreducible element of $R$ will be called \emph{homoge\-neous-irreducible} or \emph{h-irreducible} for short. 

The ring $R$ is called \emph{graded factorial} with respect to $A$, or for short \emph{$A$-factorial}, if the factorisation of any homogeneous element into h-irreducible elements is unique up to reordering and multiplication with invertible homogeneous elements.
\end{definition}

\begin{proposition}[{\cite[Thm.~1.5.3.7]{Coxrings}}]
If $X$ is a Mori dream space, then $\RR(X)$ is graded factorial with respect to $\Cl(X)$.
\end{proposition}

\subsection*{From Mori dream spaces to stacks}

In the following, we give the definitions and basic properties of Mori dream stacks as introduced in \cite{HM}.
The difference here is that we focus solely on quotient stacks, and that we do not ask for smoothness as in the definition of general Mori dream stacks in [loc.~cit.], which is replaced by graded factoriality of the Cox ring. 

For stacks, we define $\RR(\XX) = \bigoplus_{L \in \Pic(\XX)} H^0(\XX,L)$, using $\Pic(\XX)$ instead of $\Cl(\XX)$. The product structure is similarly defined to the case of varieties, the details are spelt out in \cite[\S 2]{HM}.

\begin{definition}
\label{def:mdquotientstack}
Let $\XX$ be an algebraic stack separated and of finite type over $\kk$.
We call $\XX$ a \emph{Mori dream quotient stack} (or \emph{\MD-quotient stack} for short) if 
\begin{enumerate}
\item $\Pic(\XX)$ is finitely generated as a $\ZZ$-module, 
\item $\RR(\XX)$ is finitely generated as a $\kk$-algebra,
\item $\RR(\XX)$ is graded factorial and normal,
\item $H^0(\Spec \RR(\XX) \setminus V,\OO^*) = \kk^*$,
\item $\XX$ has the following representation as a quotient stack
\[
\XX = \stackquot{\Spec \RR(\XX) \setminus V}{\Hom(\Pic(\XX),\kk^*)},
\]
where the \emph{irrelevant locus} $V = V_{\XX}$ is the vanishing locus of the irrelevant ideal
\[
\Jirr(\XX) = ( f \in \RR(\XX) \text{ homogeneous } \mid \XX_f \text{ has an affine coarse moduli space}).
\]
\end{enumerate}
\end{definition}

\begin{remark}
\label{rem:KKV_sequence}
For an algebraic quotient stack $\XX = [Z/G]$ with $Z$ normal, we have the following exact sequence by \cite[\S 2]{Knop-Kraft-Vust} 
\[
\arraycolsep=0.5ex
\begin{array}{rcccccccc}
1 &\longrightarrow& H^0(\XX, \OO^*_{\XX}) &\longrightarrow& H^0(Z, \OO^*_Z) &\longrightarrow& \Hom(G, \kk^*) &\longrightarrow \\
&\longrightarrow& \Pic(\XX) &\longrightarrow& \Pic(Z) &\longrightarrow& \Pic(G)
\end{array}
\]
which implies that also $H^0(\XX,\OO_\XX^*) = \kk^*$ for an \MD-quotient stack,
using that $H^0(\Spec \RR(\XX) \setminus V,\OO^*)=\kk^*$.
\end{remark}

The different definition of the Cox ring using the Picard group is partly justified by the proposition below.

\begin{definition} \label{def.assstack}
Let $X$ be a Mori dream space. We call
\[
\Xass = \stackquot{\Spec \RR(X) \setminus V}{\Hom(\Cl(X),\kk^*)}
\]
the \emph{canonical \MD-stack} of $X$.
\end{definition}

Note there is $\pi\colon \Xass \to X$ making $X$ the coarse moduli space of $\Xass$.

\begin{remark}
Actually our canonical \MD-stack is \emph{a} canonical stack of $X$ as defined in the literature (see~\cite[\S 1.2]{Fantechi-etal}). 
If $\Xass$ is smooth, then it is \emph{the} smooth canonical stack of $X$ according to the definition of~\cite[Def.~4.4]{Fantechi-etal}.
\end{remark}

\begin{proposition}
\label{prop:Xass}
Let $X$ be a Mori dream space with $H^0(\Spec \RR(X) \setminus V, \OO^*)=\kk^*$ 
and $\Pic(\Spec \RR(X) \setminus V)=0$. 
Then, its canonical \MD-stack $\Xass$ is an \MD-quotient stack, i.e.\ $\RR(\Xass) = \RR(X)$ and $\Pic(\Xass) = \Cl(X)$, so
\[
\Xass = \stackquot{\Spec \RR(\Xass) \setminus V}{\Hom(\Pic(\Xass),\kk^*)}
\]
\end{proposition}

\begin{proof}
First we check that $\Pic(\Xass) = \Cl(X)$.
Consider the sequence from \autoref{rem:KKV_sequence} for the quotient stack $\Xass$. 
The assumption $H^0(\Spec \RR(X) \setminus V, \OO^*)=\kk^*$ implies that the first arrow in the sequence is an isomorphism. Now $\Pic(\Spec \RR(X) \setminus V)=0$ implies the desired equality.

Finally, we have to prove that $\RR(X) = \RR(\Xass)$. Since $\pi\colon \Xass \to X$ is the map to the coarse moduli space, there is a natural isomorphism $\OO_X \cong \pi_* \OO_{\Xass}$. Using the projection formula, we get $H^0(\Xass,\OO(D)) = H^0(X,\OO(D))$ for any divisor.

By \cite[Thm.~1.5.1.1 \& Prop.~1.5.3.2]{Coxrings}, the Cox ring of any Mori dream space is graded factorial and normal, therefore $\RR(\Xass)$ as well.
\end{proof}

\begin{remark}
\label{rem:complete}
The technical assumption $H^0(\Spec \RR(X) \setminus V, \OO^*)=\kk^*$ is fulfilled for example if $H^0(X,\OO)=\kk$, by \cite[Prop~1.5.2.5]{Coxrings}.
Especially it holds as soon as $X$ is complete.

The assumption $\Pic(\Spec \RR(X) \setminus V)=0$ holds whenever $\RR(X)$ is factorial. Indeed, factoriality implies $\Cl(\Spec \RR(X)) = 0$ and, since the codimension of the irrelevant locus is at least 2, $\Cl(\Spec \RR(X) \setminus V) = \Cl(\Spec \RR(X))= 0$, hence also the Picard group vanishes.
\end{remark}

We recall the definition of two root constructions for algebraic stacks as found in \cite{Cadman} and 
\cite{Abramovich-Graber-Vistoli} and we describe them more explicitly in the case of \MD-quotient stacks.

\begin{definition}
Let $\XX$ be an algebraic stack and $D$ an effective divisor on $\XX$.
Let $n$ be a positive integer.
Consider the fibre product
\[
\xymatrix@R=3ex@C=1em{
\XX \times_{[\AA^1/\kk^*]} [\AA^1/\kk^*] \ar[r] \ar[d] & [\AA^1/\kk^*] \ar[d]^{n}\\
\XX \ar[r]_{D} & [\AA^1/\kk^*]
}
\]
where the lower map corresponds to the divisor and the  right map is induced by $p \mapsto p^n$.
The stack $\sqrt[n]{D/\XX} \coloneqq \XX \times_{[\AA^1/\kk^*]} [\AA^1/\kk^*]$
is the \emph{$n$-th root of the divisor $D$ on $\XX$}.
\end{definition}

In the following definition, we use the shorthand $\BB G$ to denote the quotient stack $[\pt/G]$, where $G$ can be an arbitrary group.

\begin{definition}
Let $\XX$ be an algebraic stack and $L$ a line bundle on $\XX$.
Let $n$ be a positive integer.
Consider the fibre product
\[
\xymatrix@R=3ex@C=1em{
\XX \times_{\BB \kk^*} \BB \kk^* \ar[r] \ar[d] & \BB \kk^* \ar[d]^{n}\\
\XX \ar[r]_{L} & \BB \kk^*
}
\]
where the lower map corresponds to the line bundle and the right map is induced by $p \mapsto p^n$.
The stack $\sqrt[n]{L/\XX} \coloneqq \XX \times_{\BB \kk^*} \BB \kk^*$ is the \emph{$n$-th root of the line bundle $L$ on $\XX$}.
\end{definition}

\begin{definition}
Let $\XX$ be an algebraic stack. We say that $\XX'$ is \emph{obtained by roots} from
$\XX$, if $\XX'$ is the result of performing a finite sequence of root constructions
with divisors or line bundles starting from $\XX$.
\end{definition}

\begin{proposition}
\label{prop:roots-preserve-MD}
Let $\XX$ be an \MD-quotient stack and $D$ a prime divisor on $\XX$ corresponding to a line bundle $L$ and a section $s \in \RR(\XX)$.
Let $n\in\NN$.
\begin{enumerate}
 \item The stack $\XX' = \sqrt[n]{D/\XX}$ is an \MD-quotient stack 
where $\RR(\XX')$ and $\Pic(\XX')$ fit into the pushout diagrams 
\[
\xymatrix@R=1.25ex@C=1.25em{
\kk[z] \ar[r]^{z \mapsto s} \ar[dd]^{z \mapsto z^n} & \RR(\XX) \ar[dd] &&
 \ZZ \ar[r]^{1 \mapsto [D]} \ar[dd]^{1\mapsto n} & \Pic(\XX) \ar[dd]\\ \\
\kk[z] \ar[r]_-{z \mapsto z}                        & \RR(\XX') \ar@{}[d]|-{\rotatebox{-90}{$\displaystyle\coloneqq$}} &&
 \ZZ \ar[r]                                     & \Pic(\XX') \ar@{}[d]|-{\rotatebox{-90}{$\displaystyle\coloneqq$}} \\
                                                    & \quot{\RR(\XX)[z]}{(z^n-s)} &&
                                                    & \quot{\Pic(\XX) \oplus \ZZ}{\ZZ \cdot (\OO(D),-n)}
}
\]
and $\RR(\XX')$ is $\Pic(\XX')$-graded by setting $\deg(z) = (0,1)$.
\item The stack  $\XX' = \sqrt[n]{L/\XX}$ is an \MD-quotient stack 
where $\RR(\XX')=\RR(\XX)$ and $\Pic(\XX')$ fits in the right diagram of (1). 
The grading of $\RR(\XX')$ is induced by $\Pic(\XX) \into \Pic(\XX')$.
\end{enumerate}
\end{proposition}

\begin{remark}
\label{rem:space_to_stack}
If $X$ is a Mori dream space with $H^0(\Spec \RR(X) \setminus V,\OO^*)=\kk^*$, then $\RR(X)= \RR(\Xass)$ is graded factorial.
As a consequence, if $\Xcan$ is an \MD-quotient stack, then any stack built from it by a sequence of roots of prime divisors or line bundles is again an \MD-quotient stack.
\end{remark}

\begin{proof}[Proof of \autoref{prop:roots-preserve-MD}]
The proof carries over from \cite[Thm.~2.9, Cor.~2.11 and Thm.~2.12]{HM}. No smoothness is needed and the irrelevant locus is defined differently here, for the treatment of that see \autoref{lem:algebra-map}.
The normality of $\RR(\XX')$ follows from the subsequent \autoref{lem:cover-normal}, applied to $U = \Spec \RR(\XX)$ and $U' = \Spec \RR(\XX')$.
Now we are left to show that $\RR(\XX')$ is still graded factorial in case (1), which is \autoref{lem:graded_ufd} below.
\end{proof}

\begin{lemma}
\label{lem:algebra-map}
Let $\XX_i = \stackquot{\Spec R_i \setminus V(J_i)}{\Hom(G_i,\kk^*)}$ with $i=1,2$ be two quotient stacks with 
\begin{itemize}
\item $R_i$ graded factorial with respect to $G_i$;
\item $J_i$ the corresponding irrelevant ideal;
\item $G_1 \into G_2$ of finite index such that $R_1 = \bigoplus_{g \in G_1} (R_2)_g$.
\end{itemize}
Then the graded morphism $R_1 \into R_2$ induces a surjective map $\XX_2 \onto \XX_1$.
\end{lemma}


\begin{proof}
For $i=1,2$ let $\tilde\XX_i = [\Spec R_i / \Hom(G_i, \kk^*)]$ and $\pi\colon \tilde\XX_2 \to \tilde\XX_1$ be the canonical (surjective) map. We only need to prove that $\pi^{-1}(V(J_1)) = V(J_2)$.

First, recall that the coarse moduli space of a stack of the form $[\Spec S/T]$ is $\Spec S^T$. 
We claim that, if $r_2 \in J_2 \subset R_2$ is $G_2$-homogeneous, then there exists a $G_1$-homogeneous $r_1  \in R_1$, such that the coarse moduli spaces of $(\tilde\XX_2)_{r_2} $ and of $(\tilde\XX_1)_{r_1}$ coincide. The claim easily implies the inclusion $\pi^{-1}(V(J_1)) \subset V(J_2)$.
As a consequence, $\XX_1$ and $\XX_2$ have the same coarse moduli space $X$.

The coarse moduli space of $(\tilde\XX_2)_{r_2}$ is $\Spec (R_2[r_2^{-1}])^{G_2}$.
Any element in this ring can be written as $p=\frac{h}{r_2^l}$, where both $h$ and $r_2^l$ have the same degree in $G_2$.
As $G_1 \subset G_2$ has finite index, there is a homogeneous $q \in R_2$, such that the degrees of $hq$ and $r_2^lq$ are in $G_1$, and therefore these elements lie in $R_1$.
So if we consider $r_1 = r_2^lq$, then $p = \frac{hq}{r_1} \in (R_1[r_1^{-1}])^{G_1}$. 
Hence the coarse moduli spaces of $(\tilde\XX_2)_{r_2} $ and of $(\tilde\XX_1)_{r_1}$ coincide as claimed.

As the map $\XX_2 \to X$ to the coarse moduli space factors through $\pi\colon\XX_2 \onto \XX_1$, the converse inclusion $V(J_2) \subset \pi^{-1}(V(J_1))$ is easy to see.
\end{proof}

\begin{lemma}
\label{lem:cover-normal}
Let $\pi\colon U' \to U$ be a cyclic cover branched at a (possibly singular) prime divisor $D$. If $U$ is normal, so is $U'$.
\end{lemma}

\begin{proof}
We show the normality using Serre's criterion.

\step{R1} Consider $\tilde U = U_{\sm} \setminus \sing(D)$, whose complement is of at least codimension two in $U$, as $U$ is normal. Then the restriction $\pi^{-1}(\tilde U) \to \tilde U$ is again a cyclic cover and as $D \cap \tilde U$ is smooth, therefore $\pi^{-1}(\tilde U)$ is smooth by \cite[Cor~1.23]{HM}. 

\step{S2} 
We follow here an argument from the proof of \cite[Lem.~1.2]{Alexeev-Pardini}, which we repeat for the convenience of the reader. 

Without loss of generality, we may assume that $U = \Spec A$ and $D$ given by a prime $p \in A$. Then $U' = \Spec A'$ with $A' = A[z]/(z^n-p)$ for some $n\in\NN$.
 So $A' = \bigoplus_{i=0}^{n-1} A\cdot z^i$ as an $A$-module. As each summand $A \cdot z^i$ is saturated in codimension $2$, $U' = \Spec_{\OO_U} \OO_{U'}$ is an $S_2$-variety.
\end{proof}

\begin{lemma}
\label{lem:graded_ufd}
Let $\XX$ be an \MD-quotient stack, in particular its Cox ring $\RR(\XX)$ is $\Pic(\XX)$-factorial.
Let $s \in \RR(\XX)$ and $D$ be the corresponding divisor. 
Then for $\XX' = \sqrt[n]{D/\XX}$, the ring $\RR(\XX')$ is $\Pic(\XX')$-factorial if and only if $D$ is a prime divisor.
\end{lemma}

\begin{proof}
First note that $D$ is a prime divisor if and only if $s$ is $\Pic(\XX)$-irreducible.
In the following we use the notations
\[
R \coloneqq \RR(\XX), A \coloneqq \Pic(\XX), R' = \RR(\XX') \text{ and } A' = \Pic(\XX').
\]
Assume that $s$ is not $A$-irreducible in $R$, so there is a non-trivial factorisation $s=s_1 s_2$ into $A$-homogeneous elements. Since these are also $A'$-homogeneous, we have by $z^n = s_1 s_2$ two homogeneous factorisations for $s$ in $R[z]/(z^n-s)$. So $R'$ is not $A'$-factorial.

Now suppose that $R'$ not $A'$-factorial.
If $t_1 t_2 = t_3 t_4$ are two different homogeneous factorisations in $R'$,
then we can write $t_i = z^{m_i} t'_i$ where $t'_i \in R$ and $m_i$ is chosen to be maximal. Therefore, $s$ does not divide $t'_i$.
Then we obtain the equality
\[
z^{m_1+m_2} t_1' t_2' = z^{m_3+m_4} t_3' t_4'
\]
where $m_1+m_2 = m_3+m_4 \ \mod n$. So we assume without loss of generality that $m_1+m_2 = m_3+m_4 + ln$ with $l \in \NN_0$.
Reducing both sides by $z^{m_3+m_4}$ we obtain
\[
z^{ln} t_1't_2' = s^{l} t_1't_2' = t_3't_4'
\]
The latter equality are two homogeneous factorisations in $R$. By the maximality of the $m_i$'s this implies that $s$ cannot be $A$-irreducible.
\end{proof}

\subsection*{More Mori dream stacks and maps}

Let $Y$ be a Mori dream space. For a subgroup $\KK \subset \Cl(Y)$, we can define 
the $\KK$-graded
\[
\RR(Y;\KK) = \bigoplus_{[D] \in \KK} H^0(Y,\OO(D)).
\]

\begin{remark}
\label{rem:divquot_for_Y}
If we consider $\Pic(Y) \subset \KK \subset \Cl(Y)$ then it still holds $Y = \quot{\Spec \RR(Y;\KK) \setminus V }{ \Hom(\KK,\kk^*)}$, see for example \cite[\S 4.2.1]{Coxrings}.
On the other hand, the quotient stack
\[
\YY_\KK \coloneqq \stackquot{\Spec \RR(Y;\KK) \setminus V}{\Hom(\KK,\kk^*)}
\]
will be different from $\YY^{\ass}$.
\end{remark}

\begin{lemma}
Let $Y$ be a $\QQ$-factorial variety and $\KK$ a group with $\Pic(Y) \subset \KK \subset \Cl(Y)$.
If $\RR(Y)$ is a finitely generated $\kk$-algebra, then $\RR(Y;\KK)$ as well.
Moreover, $\RR(Y)$ is a finitely generated $\RR(Y;\KK)$-module.
\end{lemma}

\begin{proof}
We abbreviate $\RR \coloneqq \RR(Y)$ and $\RR' \coloneqq \RR(Y;\KK)$.
By \cite[Prop.~7.8]{Atiyah-Macdonald}, we can restrict ourselves to show that $\RR$ becomes a finitely generated $\RR'$-module by the inclusion $\RR' \into \RR$.
Let $r_1,\ldots,r_l \in \RR$ be generators of $\RR$ as a $\kk$-algebra. Since $Y$ is $\QQ$-factorial, there are positive integers $b_i$ such that $r_i^{b_i} \in \RR'$. Then $\RR$ is generated by products of the $r_i^j$ with $1\leq i \leq l$ and $1\leq j < b_i$ as an $\RR'$-module.
\end{proof}

\begin{definition}
Let $G$ and $H$ be abelian groups, $R$ and $S$ $\kk$-algebras graded with respect to $G$ and $H$, respectively. Let $\alpha\colon G \to H$ be a group homomorphism. A morphism of $\kk$-algebras $\phi\colon R \to S$ is called \emph{homogeneous} 
if $\phi(R_g) \subseteq S_{\alpha(g)}$, for any $g \in G$. 

Similarly, if $G$ and $H$ act on spaces $U$ and $V$, respectively, we call $f\colon U\to V$ \emph{equivariant} if $f \circ \sigma_g = \sigma_{\alpha(g)} \circ f$, where $\sigma$ denotes the action.
\end{definition}

\begin{remark}
Any equivariant map $f \colon U \to V$ induces maps $U/G \to V/H$ and $[U/G] \to [V/H]$.
Especially, for a Mori dream space $Y$ and subgroups $\KK \subset \KK' \subset \Cl(Y)$ we obtain $\YY_{\KK'} \onto \YY_\KK$.
\end{remark}

\section{The Cox lift}
\label{sec:coxlift}

Suppose $\phi\colon X \to Y$ is a map between  Mori dream spaces which we want to lift to a map $\phi^*\colon\RR(Y) \to \RR(X)$.
One important property of a Cox ring is that any effective prime divisor $D$ on $Y$ is given by a global homogeneous equation $r \in \RR(Y)$.
Ideally, $\phi^*(r)$ would be just the equation of the pulled-back divisor $\phi^*(D)$. But if $D$ is only a Weil divisor and not necessarily Cartier, then we have no well-defined pull-back of $D$. As $Y$ is $\QQ$-factorial, some multiple $nD$ is Cartier.
Say that $\phi^*(nD)$ is given by some $s \in \RR(X)$. The key idea is to introduce the necessary roots $\sqrt[n]{s}$ to $\RR(X)$ if $s$ is non-zero or otherwise to take a line bundle root of $\OO(\phi^*(nD)$ to obtain a map $\Phi^*\colon \RR(Y) \to \RR(\XX)$ of Cox rings.

The following lemma can also be seen as an easy consequence of \cite[Prop.~5.6]{Berchtold-Hausen}, which essentially states that the assignment $Y \mapsto \RR(Y;\Pic(Y))$ (together with the irrelevant ideal) is functorial.
Still we present a proof here for the convenience of the reader.

\begin{lemma}
\label{lem:inducedpicmap}
If $\phi\colon X\to Y$ is a map of  Mori dream spaces,
then there is an induced homogeneous morphism $\Phi_0^*\colon\RR(Y;\Pic(Y)) \to \RR(X)$ and consequently an equivariant map $\Phi_0\colon \Spec \RR(X) \setminus V \to \Spec \RR(Y;\Pic(Y)) \setminus V$.
\end{lemma}

\begin{proof}
\newcommand{\phiCK}{\phi_{\div}^*}
\newcommand{\phiPicCl}{\phi^*}
\newcommand{\PhiS}{\tilde\Phi_0^*}
\newcommand{\PhiR}{\Phi_0^*}
Let $C^0_Y \into C_Y \onto \Pic(Y)$ be a free presentation of $\Pic(Y)$ where $C_Y = \bigoplus_{i=1}^l \ZZ \cdot D_i$ with $D_i \in \CDiv(Y)$ (which need not to be effective).
Without loss of generality $\im(\phi)$ is not contained in $\supp D_i$ for any $i$.
Let us assume the contrary that $\im(\phi) \subset \supp D_i$ for some $i$.
The $D_i$ is locally principal 
at any given point, for example $y \in \im(\phi)$.
Therefore $D_i$ is given as $\div(s)$ on an open neighbourhood $U$ of $y$.
If we replace $D_i$ by the linearly equivalent divisor $\tilde D_i = D_i - \div(s)$,
then $\supp \tilde D_i \cap U = \emptyset$, especially $\im(\phi) \not\subset \supp \tilde D_i$.

So we can define $\phiCK \colon C_Y \to \WDiv(X)$ whose image is contained in some free subgroup $K_X \subset \WDiv(X)$ that generates $\Cl(X)$. 
Obviously this map carries $C_Y^0$ to $K^0_X$ as it is just the composition of the rational functions with $\phi$. Therefore we get the following commutative diagram:
\[
\xymatrix@R=3ex@C=1em{
C^0_Y \ar@{^(->}[r] \ar[d] & C_Y \ar@{->>}[r] \ar[d]^{\phiCK} & \Pic(Y) \ar@{-->}[d]^{\phi^*}\\
K^0_X \ar@{^(->}[r] & K_X \ar@{->>}[r] & \Cl(X)
}
\]
where the dashed arrow is the one induced from $\phiCK$ on the quotient.

Next, we claim that we have a well-defined ring homomorphism
\[
H^0(Y,\SS_Y) = \bigoplus_{D \in C_Y} H^0(Y,\OO_Y(D)) \xto{\PhiS} H^0(X,\SS_X) = \bigoplus_{D \in K_X} H^0(X,\OO_X(D))
\]
which is graded with respect to $\phiCK \colon C_Y \to K_X$. Here we consider the global sections as rational functions, so their pullback is not problematic as $\im(\phi) \not\subset \supp D$ for all $D \in C_Y$.

\newcommand{\newchiY}{\widetilde{\chi}_Y}
Choose now maps $\chi_X: K_X^0 \to \kk(X)^*$ and $\chi_Y: C_Y^0 \to \kk(Y)^*$, as in the construction of the Cox ring.
Then there exist $\alpha_i \in \kk^*$, such that $\chi_X(\phiCK(D_i))= \alpha_i \cdot  \phi^*(\chi_Y(D_i))$, for $1 \leq i\leq l$. 
Define $\newchiY$ by setting $\newchiY(D_i) =   \alpha_i \cdot \chi_Y(D_i)$, then the following diagram commutes     
\[
\xymatrix@R=3ex@C=2em{
C^0_Y \ar[r]^-{\newchiY} \ar[d]_{\phiCK} & \kk(Y)^* \ar[d]^{\phi^*}\\
K^0_X \ar[r]^-{\chi_X}                 & \kk(X)^*
}
\] 
Using $\chi_X$ and $\newchiY$ to define the homogeneous ideals $\II_X$ and $\II_Y$, 
the map $\PhiS\colon H^0(Y,\SS_Y) \to H^0(X,\SS_X)$ restricts to a map $H^0(Y,\II_Y) \to H^0(X,\II_X)$ and therefore induces the map 
$\Phi_0^*\colon \RR(Y;\Pic(Y)) \to \RR(X)$.

To get $\Phi_0$, we only need to see that the image of the irrelevant locus of $X$ is contained in the one of $Y$.
We cover $Y$ by the affine sets $Y_j \coloneqq Y_{g_j} = \{g_j \neq 0\}$, where we choose the $g_j$ as homogeneous generators of the irrelevant ideal of $Y$. In the same way cover the preimage of $Y_j$ under $\phi$ with affine sets $X_{ij} \coloneqq X_{f_{ij}}$, where the $f_{ij}$ are in the irrelevant ideal of $X$. This is possible by \cite[Prop.~1.6.3.3]{Coxrings} as these $X_{ij}$ form a base of the topology, additionally even $\sqrt{(f_{ij} \mid i,j )} = \Jirr(X)$ holds.

Especially we can restrict $\phi$ for all $i,j$ to $X_{ij} \to Y_j$. As $X$ and $Y$ are good quotients, we get the following diagram of rings
\[
\xymatrix@R=3ex@C=5em{
**[r]\RR(Y;\Pic(Y))[g_j^{-1}] \ar@{-->}[r] & **[r]\RR(X)[f_{ij}^{-1}] \\
**[r]\RR(Y;\Pic(Y))[g_j^{-1}]^{\Hom(\Pic(Y),\kk^*)} \ar@<-8ex>@{^(->}[u] \ar[r]_-{\phi^*} & **[r]\RR(X)[f_{ij}^{-1}]^{\Hom(\Cl(X),\kk^*)} \ar@<-3ex>@{^(->}[u]
}
\]
Note that $\RR(X)[f_{ij}^{-1}]$ is the coordinate ring of $\tilde X_{ij} \coloneqq \pi^{-1}(X_{ij})$ using that $\pi\colon \Spec\RR(X) \setminus V \onto X$ is a good quotient map.
By the choices of the open sets, $\Phi_0^*(g_j)$ vanishes only outside of $\tilde X_{ij}$, hence is an invertible function in $\RR(X)[f_{ij}^{-1}]$.
So the dashed arrow exists, and therefore $\Phi_0^*$ induces $\Phi_0\colon \Spec \RR(X)\setminus V \to \Spec \RR(Y;\Pic(Y))\setminus V$.
\end{proof}

The map in the previous lemma depends on choices, but any two such maps are equivalent up to non-unique isomorphisms of the respective Cox rings.

\begin{theorem}
\label{thm:lift}
Let $\phi\colon X \to Y$ be a map between Mori dream spaces
with $H^0(\Spec \RR(X) \setminus V,\OO^*)=\kk^*$ and $\Pic(\Spec \RR(X) \setminus V) = 0$.
There is an \MD-quotient stack $\XX$ with coarse moduli space $X$ with the
following properties:
\begin{enumerate}
\setlength{\multicolsep}{0ex}
\item \label{itm:lift_root} The stack $\XX$ is built from $\Xass$ by root constructions
with prime divisors and line bundles, in particular the Cox ring $\RR(\XX)$ is graded factorial.
\item \label{itm:lift_quotient} There is a homogeneous morphism $\Phi^*\colon
\RR(Y) \to \RR(\XX)$ such that
the following diagrams commute: 
\[
\xymatrix@R=3ex@C=1em{
\RR(Y) \ar[r]^{\Phi^*} & \RR(\XX) &&
 \XX \ar[r]^-\Phi \ar[d] & \YY^{\ass} \ar[d]\\
\RR(Y; \Pic(Y)) \ar[u] \ar[r]_-{\phi^*} & \RR(X) \ar[u] &&
 X \ar[r]_\phi & Y
}
\]
\item \label{itm:lift:3} 
The stack $\XX$ is minimal:
given any \MD-quotient stack $\XX'$ that satisfies \eqref{itm:lift_quotient},
then the map $\XX'\to X$ to the coarse moduli space factors through $\XX$.
\end{enumerate}
We call $\Phi\colon \XX \to \YY^{\ass}$ the \emph{Cox lift} of the map $\phi\colon X \to Y$.
\end{theorem}

\begin{proof}
We first construct an \MD-stack $\XX$ inductively, which satisfies  \eqref{itm:lift_root} and~\eqref{itm:lift_quotient}.

\step{Setup}
Let $\KK_0 = \Pic(Y)$ and define
\[
\YY_0 \coloneqq \stackquot{\Spec \RR(Y;\KK_0) \setminus V}{\Hom(\KK_0,\kk^*)}.
\]
Consider the map of \autoref{lem:inducedpicmap}
\[
\Phi_0^* \colon \RR(Y;\KK_0) \to \RR(X) = \RR(\Xass), 
\]
inducing the map
$\Phi_0 \colon \XX_0 \coloneqq \Xass \to \YY_0$.
Since $\RR(\XX_0)$ is graded factorial,
if $\KK_0 = \Cl(Y)$, then the map $\Phi_0$ is already the Cox lift.

Otherwise, we proceed by induction. So let us assume, we have constructed a map $\Phi_i \colon \XX_i \to \YY_i$ given by a homogeneous $\Phi_i^* \colon \RR(Y;\KK_i) \to \RR(\XX_i)$, where $\XX_i$ is obtained by roots from $\XX_{i-1}$ and whose Cox ring $\RR(\XX_i)$ is graded factorial.

\step{Step (A): Add a Weil divisor and its generating sections}
If $\KK_i \neq \Cl(Y)$, choose a divisor class $[D_{i+1}] \in \Cl(Y) \setminus \KK_i$ such that there is a prime $p \in \NN$ such that $p[D_{i+1}] \in \KK_i$. Define $\KK_{i+1} \coloneqq \genby{\KK_i,[D_{i+1}]}$ and set
\[
\YY_{i+1} \coloneqq \stackquot{ \Spec \RR(Y;\KK_{i+1}) \setminus V}{\Hom(\KK_{i+1},\kk^*)}.
\]
Then we can write $\RR(Y,\KK_{i+1})$ as an $\RR(Y;\KK_i)$-algebra
\[
\RR(Y;\KK_{i+1}) = \RR(Y;\KK_i)[r_1,\ldots,r_l]
\]
where the $r_j$ form a minimal set of homogeneous generators, i.e.\ each $r_j$ is in $H^0(Y,F_j)$ for some $F_j \in \KK_{i+1} \setminus \KK_i$.

By construction, we get that $r_{j}^p \in \RR(Y;\KK_i)$ for all $j$.
Therefore, $\Phi_i^*(r_j^p) \in \RR(\XX_i)$ is well-defined for each $j$.

\step{Step (B): If some pullbacks are non-zero, take appropriate roots along divisors}
Suppose there is at least a non-zero $\Phi_i^*(r_j^p)$.
As $\RR(\XX_i)$ is graded factorial, there is a set $\{\qq_1,\ldots,\qq_m\}$ of h-irreducible elements such that any non-zero $\Phi_i^*(r_j^p)$ can be written as
\[
\Phi_i^*(r_j^p) = \prod_{l=1}^m \qq_{l}^{a_{l,j}}
\]
where the $a_{l,j} \geq 0$ are integers.
For each of these factors $\qq_l$, define 
\[
b_l = \frac{p}{\gcd(p,a_{l,1},\ldots,a_{l,m})} = 
\begin{cases}
1   & \text{$p$ divides $a_{l,j}$ for all $j$;}\\
p & \text{else.}
\end{cases}
\]

Let $E_1,\ldots,E_{m}$ be the prime divisors on $\XX_i$ corresponding to $\qq_1,\ldots,\qq_m$.
Define $\XX_{i+1}$ as the following root over $\XX_i$:  
\[
\XX_{i+1} = \sqrt[(b_1,\ldots,b_m)]{(E_1,\ldots,E_m)/\XX_i},
\]
in particular, the Cox ring 
\[
\RR(\XX_{i+1}) = \RR(\XX_i)[z_1,\ldots,z_m]/(z_1^{b_1}-\qq_1,\ldots,z_m^{b_m}-\qq_m)
\]
is graded factorial by \autoref{lem:graded_ufd}. 
Here we mean subsequent root constructions along the $E_l$,
so that the $E_l$ stay prime when lifted to the previous root construction. Geometrically, only stabilisers will be added along the previous divisors (and outside the stack remains unchanged), therefore the $E_l$ stay prime as they are pairwise different.
We also remark that if $b_l=1$, then the according root construction keeps the stack unchanged.

Next we define the morphism
\[\Phi_{i+1}^* \colon  \RR(Y;\KK_{i+1}) = \RR(Y;\KK_i)[r_1,\ldots,r_n]  \to  \RR(\XX_{i+1})\]
by extending $\Phi_i^*$ with
\[ r_j  \mapsto  {\displaystyle \xi^{\alpha_j} \prod_{l=1}^m \sqrt[p]{\qq_l^{a_{l,j}}}} \]
where by slight abuse of notation we mean
\[
\sqrt[p]{\qq_l^{a_{l,j}}} = z_l^{\frac{b_l a_{l,j}}p}.
\] 
Moreover, there is the choice of an $p$-th root of unity $\xi^{\alpha_j}$ for each $\Phi_{i+1}^*$, where $\xi$ is a primitive root.
Obviously, this turns
\[
\xymatrix@R=3ex@C=2em{
\RR(Y;\KK_{i+1}) \ar[r]^{\Phi_{i+1}^*} & \RR(\XX_{i+1})\\
\RR(Y;\KK_i) \ar[r]^{\Phi_i^*} \ar[u] & \RR(\XX_i) \ar[u]
}
\]
into a commutative diagram of homogeneous morphisms, provided that $\Phi_{i+1}^*$ preserves the relations among the generators.

\smallskip 

First we will determine the elements $\xi^{\alpha_j}$, as they cannot be chosen arbitrarily. 
We want $\Phi_{i+1}^*$ to coincide with $\Phi_i^*$ on $\RR(Y;\KK_i)$, which can be checked on monomials $\mathbf r^{\mathbf c} = r_1^{c_1}\cdots r_n^{c_n}$.
As $\Phi_i^*(\mathbf r^{\mathbf c}) \in \RR(\XX_i) \subset \RR(\XX_{i+1})$,
we get by graded factoriality in $\RR(\XX_{i+1})$ that
\[
\Phi_i^*(\mathbf r^{\mathbf c}) = \xi^{\alpha_{\mathbf c}} \prod_{j=1}^n \prod_{l=1}^m \left(\sqrt[p]{\qq_l^{a_{l,j}}}\right)^{c_j}
\]
As we ask for equality
$\Phi_i^*(\mathbf r^{\mathbf c}) = \Phi_{i+1}^*(\mathbf r^{\mathbf c})$,
we get for each monomial $\mathbf r^{\mathbf c} \in \RR(Y;\KK_i)$ a linear equation for the $\alpha_j$ in the field $\ZZ_p$:
\[
\alpha_{\mathbf c} = c_1 \cdot \alpha_1 + \cdots + c_n \cdot \alpha_n.
\]
Up to multiplication with $p$-th powers $r_j^p$ (which does not change the exponent of $\xi$), all monomials $\mathbf r^{\mathbf c} \in \RR(Y;\KK_i)$ are in bijection to the $(n-1)$-dimensional $\ZZ_p$-vector space
\[
\Big\{ \mathbf c \in \ZZ_p^n \mid \sum_{j=1}^n c_j \cdot \deg_{i+1}(r_j) = 0  \text{ in } \ZZ_p \Big\}
\]
Hence only the linear relations for $(n-1)$ monomials corresponding to a basis of this vector space have to be fulfilled. Choose any of these $p$ solutions.

\smallskip

With $\xi^{\alpha_j}$ now fixed, we can complete the check of the well-definedness of $\Phi_{i+1}^*$. Consider an arbitrary homogeneous relation $f( r_1,\ldots,r_n ) =0$ among the generators of $\RR(Y;\KK_{i+1}) \subset \RR(Y)$.
Since $\RR(Y)$ is graded factorial, $f(r_1,\ldots,r_n)=0$ if and only if $r_j^a f(r_1,\ldots,r_n)=0$ for any $a \in \NN$. By choosing $a$ appropriately, latter is a relation in $\RR(Y;\KK_i)$, i.e.\ a sum of monomials $r_1^{c_1}\cdots r_n^{c_n}$ with coefficients in $\RR(Y;\KK_i)$ such that $\sum c_j \deg_{i+1}(r_j) = 0 \in \KK_{i+1}/\KK_i$ for all these monomials.
Now $\Phi_i^*$ is already a ring homomorphism, hence $\Phi_i^*(r_j^a f(r_1,\dots,r_n))=0$ in $\RR(\XX_i) \subset \RR(\XX_{i+1})$.
As $\Phi_i^*$ and $\Phi_{i+1}^*$ coincide on $\RR(Y;\KK_i)$, 
we can rewrite
\[
\begin{split}
0 & = \Phi_i^*(r_j^a f(r_1,\ldots,r_n))= \Phi_{i+1}^*(r_j^a f(r_1,\ldots,r_n)) \\
  & = \Phi_{i+1}^*(r_j)^{a} \cdot  \Phi_{i+1}^*( f(r_1,\ldots,r_n)) .
\end{split}
\]
By graded factoriality, $\Phi_{i+1}^*(f(r_1,\ldots,r_n))$ has to be zero, since we can choose an $r_j$ such that $\Phi_{i+1}^*(r_j) \neq 0$.

\step{Step (C): If all pullbacks are zero, take a root along a line bundle}
Suppose now that $\Phi_i^*(r_j^p) = 0$ for all $j$. Since $p[D_{i+1}] \in \KK_i$, the pullback under the map $\Phi_i$ is defined.
Let $\Phi_i^*(p[D_{i+1}])= \LL_i \in \Pic(\XX_i)$. 
We define $\XX_{i+1} \coloneqq \sqrt[p]{\LL_i/\XX_i}$ by a line bundle root and we obtain a morphism  
$\Phi_{i+1}^* \colon \RR(Y;\KK_{i+1}) \to \RR(\XX_{i+1}) = \RR(\XX_i)$ by setting 
$\Phi_{i+1}^*(r_j) = 0$. Since $\KK_{i+1} \to \Pic(\XX_{i+1})$ is given by 
$[D_{i+1}] \mapsto \sqrt[p]{\LL_i}$, the homogeneity follows.

\medskip
After performing root constructions as either in Step (B) or (C), 
we are done if $\KK_i=\Cl(Y)$. Otherwise increase $i \mapsto i+1$ and return to Step (A).

\step{Minimality}
We have built the following tower of \MD-stacks, where the vertical maps are induced from the graded morphisms using \autoref{lem:algebra-map}:
\[
\xymatrix@R=3ex@C=1em{
& X \ar^{\phi}[r] & Y\\
\Phi_0 \colon & **[l] \Xass \eqqcolon \XX_0 \ar[r] \ar[u]  & \YY_0 \ar[u]\\
\Phi_1 \colon & \XX_1 \ar[r] \ar[u] & \YY_1  \ar[u]\\
              & \vdots       \ar[u] & \vdots \ar[u]\\
**[l] \Phi = \Phi_k \colon & **[l] \XX \coloneqq \XX_k \ar[r] \ar[u] & **[r] \YY_k = \YY^{\ass} \ar[u]\\
}
\]

To prove the minimality of $\Phi$ as in \eqref{itm:lift:3}, consider the following diagram:  
\[
\xymatrix@R=3ex{
&\Xass \ar[r] & \YY_0 \\
&\XX \ar[r] \ar[u] & \YY^{\ass} \ar[u]\\
\XX' \ar@/^/[uur]^{\Theta_0} \ar@/_/[rru]_{\Psi} \ar@{.>}[ur]|{\Theta}
}
\]
where $\XX'$ fulfills \eqref{itm:lift_quotient}. 
We can define the map $\Theta$ inductively as follows. Consider a single step:
\[
\xymatrix@R=3ex{
&\XX_{i} \ar[r]^{\Phi_i} & \YY_{i} \\
&\XX_{i+1} \ar[r]^{\Phi_{i+1}} \ar[u] & \YY_{i+1} \ar[u]\\
\XX' \ar@/^/[uur]^{\Theta_{i}} \ar@/_/[rru]_{\Psi_{i+1}}\ar@{.>}[ur]|{\Theta_{i+1}}
}
\]

If $\XX_{i+1}$ is obtained from $\XX_i$ as a root over the line bundle $\LL_i$, then the pullback $\Theta_i^*$ provides a map 
$\RR(\XX_{i+1}) = \RR(\XX_i) \to \RR(\XX')$. For the homogeneity we show that there is a unique arrow $\vartheta_{i+1}$ in the following commutative diagram:
\[  
\xymatrix@R=3ex@C=1em{ 
       & \Pic(\XX_i) \ar[d] \ar@/_/[ddl]_{\vartheta_i}   & \KK_i \ar[l]\ar[d] \\
       & \Pic(\XX_{i+1}) \ar@{.>}[dl]|{\vartheta_{i+1}} & \KK_{i+1} \ar[l] \ar@/^/[dll]^{\psi_{i+1}} \\    
\Pic(\XX') & &  }
\]
By commutativity $\psi_{i+1}(p[D_{i+1}])= \vartheta_i(\LL_i)$. Recall that 
\[
\Pic(\XX_{i+1})= \quot{\Pic(\XX_i) \oplus \ZZ}{\ZZ \cdot (\LL_i,-p)},
\]
thus we can define the dotted map $\vartheta_{i+1}$ by extending $\vartheta_i$ with $\sqrt[p]{\LL_i} = (\OO, 1) \mapsto \psi_{i+1}(D_{i+1}).$

Let now $\XX_{i+1}$ be obtained from $\XX_i$ by roots of divisors. We define the map $\Psi_{i+1}^* \colon \RR(Y;\KK_{i+1}) \into \RR(Y) \xto{\Psi^*} \RR(\XX')$. 
Consequently, we have $\Psi_{i+1}^*(r_j) \in \RR(\XX')$ with
\[
\Psi_{i+1}^*(r_j)^p = \Psi_i^*(r_j^p) = \Theta_{i}^* \circ \Phi_{i}^*(r_j^p) = 
\prod_{l=1}^m \Theta_{i}^*(\qq_l^{a_{l,j}}) = \prod_{l=1}^m \Theta_{i+1}^*(\qq_l)^{a_{l,j}} 
\]
Since  $\RR(\XX')$ is graded factorial, it contains elements $w_l$,
 such that $w_l^{b_l}= \qq_l$ where $b_l$ is still as above.
These elements are unique up to multiplication with a $b_l$-th root of unity $\xi$.
We define the map $\Theta_{i+1}^* \colon \RR(\XX_{i+1}) \dashrightarrow \RR(\XX')$ by sending $z_l \mapsto w_l$ for $l=1,\ldots, m$. 
Note that any other choice $z_l \mapsto \xi w_l$ gives the same map precomposed with an automorphism of $\RR(\XX_{i+1})$.
Its homogeneity is easy to see, 
as before we define the dotted morphism $\vartheta_{i+1}\colon\Pic(\XX_{i+1}) \to \Pic(\XX')$ by $\OO(\sqrt[b_l]{E_l}) \mapsto \OO(w_l)$, 
where $\OO(w_l)$ denotes the zero divisor of the function $w_l$.
To see that the resulting $\Theta_{i+1}^*\colon\RR(\XX_{i+1})\dashrightarrow\RR(\XX')$ is well-defined, take again a relation among the generators, and apply the same argument as for the statement that $\Phi_{i+1}^*$ is a homogeneous morphism.
\end{proof}

\begin{proof}[Proof of \hyperlink{maintheorem}{Main Theorem}]
This follows immediately from \autoref{thm:lift} taking \autoref{rem:complete} into account.
\end{proof}

Let us now see how to compute the Cox lift in some examples.

\begin{example}
\label{ex:A2_mod_mu2:closer}
We apply the proof of \autoref{thm:lift} to \autoref{ex:A2_mod_mu2} from the introduction.
So we start with the map
\[
\FctArray{
\Phi_0^* = \phi^* \colon & \kk[x^2,xy,y^2] & \to & \kk[t]\\
& x^2 & \mapsto & t\\
& xy,y^2 & \mapsto & 0
}
\]
At this point we have that $\KK_0 = \Pic(Y) = \{\OO\} \neq \Cl(Y) = \{\OO,\OO(D)\}$.
Therefore, we only need to make one step and add the sections of $\OO(D)$, to get
\[
\RR(Y) = \kk[x,y] =  \RR(Y;\Pic(Y))[x,y].
\]
For any section $ax+by$ of $\OO(D)$, its square will be a Cartier section of $\OO$ and its pullback is
\[
\phi^*( (ax+by)^2 ) = \begin{cases} t  & \text{if $a\neq0$;}\\ 0 & \text{if $a=0$.}\end{cases}
\]
Hence for the generator $x$ we have to perform a $2$nd-root construction on $X$, 
whereas for the second generator $y$, there is nothing to do, since its restriction to the image of $\phi$ is zero.
This yields the stack
\[
\XX = \XX_1 =  \stackquot{\Spec \kk[t,z]/\langle t^2 - z \rangle }{ \mu_2 } = \stackquot{\Spec \kk[\sqrt{t}]}{\mu_2}.
\]
Therefore we arrive at the conclusion that the Cox lift of $\phi\colon\AA^1 \into \AA^2/\mu_2$ is
$\Phi\colon\sqrt[2]{0/\AA^1} \into [\AA^2/\mu_2]$, as already obtained in the introduction.
\end{example}

\begin{example}
As a variation of the previous example, we can consider the inclusion $\{0\} \into Y = \AA^2/\mu_2$ as the singular point.
Note that $\Cl(Y)\cong\ZZ_2$ is generated by a Weil divisor $D$ through zero.
The ring $\RR(Y;\Pic(Y))$ would be naively given by the sections of $2D$, but for the pullback (following the proof of \autoref{lem:inducedpicmap}) we consider instead the sections of $\OO_Y$.

Then the square power of any global section of the Weil divisor $D$ will restrict to zero on the image of the inclusion $\{0\} \into Y$. All these sections can be identified with sections of $\OO_Y$ and can be pulled back. As said all pullbacks vanish, so we are in the case that we take a line bundle root of the trivial line bundle on $\{0\}$.
Hence the Cox lift is
$\BB \mu_2 = [\{0\}/\mu_2] \into [\AA^2/\mu_2]$.

This fits the geometric interpretation given in \autoref{ex:A2_mod_mu2}, as $\{0\}$ is the singular point of $\AA^2/\mu_2$, and the corresponding lift is the inclusion of the origin of $[\AA^2/\mu_2]$ which has a $\mu_2$-stabiliser.

The original $\{0\}$ also fulfills the first two properties of the Cox lift but is not universal, as the map $\BB \mu_2 \into [\AA^2/\mu_2]$ does not factor over $\{0\}$.
\end{example}

\begin{example}
\label{ex:mu3}
Let $\phi\colon X \to Y$ be a map of  Mori dream spaces, where $Y = \AA^2/\mu_3$ is the $\frac13(1,2)$-singularity.
Here $\KK_0 = \Pic(Y) = \{\OO\}$ and 
\[
\RR(\YY_0) = \RR(Y;\Pic(Y)) = \kk[x^3,xy,y^3].
\]
We choose $\OO(D)$ in such a way that its global sections are generated by $x$ and $y^2$ as an $\RR(Y;\Pic(Y))$-module, likewise  $H^0(Y,\OO(2D)) = \genby{x^2,y}$.
Note that $\Cl(Y) = \{\OO,\OO(D),\OO(2D)\}$ and $\RR(Y) = \kk[x,y]$.
 
Following the proof of \autoref{thm:lift}, we lift the homomorphism
\[
\FctArray{
\Phi_0^*\colon & \kk[x^3,xy,y^3]  & \to & \RR(X) = \RR(\XX_0)\\
& x^3 & \mapsto & u \phantom{=\RR(\XX_0)}\\
& xy  & \mapsto & v \phantom{=\RR(\XX_0)}\\
& y^3 & \mapsto & w \phantom{=\RR(\XX_0)}
}
\]
with $uw-v^3=0$, to a homomorphism
\[
\Phi^* = \Phi_1^* \colon \RR(Y) = \RR(Y;\Pic(Y))[x,y] = \kk[x,y] \to \RR(\XX_1) = \RR(\XX).
\]
To define it, we consider the generators $x$ and $y$ and use that $\Phi_0^*(x^3) = u$ and $\Phi_0^*(y^3) = w$.

\step{Case $u=0=w$} In this case also the lifted map $\Phi^*$ will be the zero map, but a third line bundle root has to be performed. Hence $\RR(\XX) = \RR(X)$ graded by $\Pic(\XX) = \quot{\Cl(X) \otimes \ZZ}{\ZZ\cdot(\OO(D),-3)}  \onto \Cl(X)$.

\step{Case $u$ or $w$ is non-zero}
Here we need to introduce the necessary third roots for those h-irreducible factors of $u$ and $w$ which do not appear with a multiplicity divisible by $3$ and get
\[
\FctArray{
\Phi^* \colon & \kk[x,y] & \to & \RR(\XX) \\
 & x & \mapsto & \xi^i \sqrt[3]{u} \\
 & y & \mapsto & \xi^j \sqrt[3]{w}
}
\]
where $\xi$ is a primitive third root of unity.
Finally, note that $v = \xi^k \sqrt[3]{u}\sqrt[3]{w}$. 
As we want $v = \Phi^*(x) \cdot \Phi^*(y) = \xi^{i+j} \sqrt[3]{u} \sqrt[3]{w}$, we have to choose $i,j \in \ZZ_3$ in such a way that $i+j = k$.
\end{example}

\begin{example}
We modify the previous example to a map $\phi\colon X \to Y$ where $Y = \AA^2/\mu_4$ is the $\frac14(1,2)$-singularity.
Here $\KK_0 = \Pic(Y) = \{\OO\}$, $\RR(\YY_0) = \RR(Y;\Pic(Y)) = \kk[x^4,x^2y,y^2]$, 
$\Cl(Y) = \ZZ_4$ and $\RR(Y)= \kk[x,y]$.
Following the proof of \autoref{thm:lift}, we lift the homomorphism
\[
\FctArray{
\Phi_0^* \colon & \RR(\YY_0) = \RR(Y;\Pic(Y)) = \kk[x^4,x^2y,y^2] &\to& \RR(X) = \RR(\XX_0)\\
& \phantom{\RR(\YY_0) = \RR(Y;\Pic(Y)) = } x^4 & \mapsto & u \phantom{ = \RR(\XX_0)}\\
& \phantom{\RR(\YY_0) = \RR(Y;\Pic(Y)) = } x^2y& \mapsto & v \phantom{ = \RR(\XX_0)}\\
& \phantom{\RR(\YY_0) = \RR(Y;\Pic(Y)) = } y^2 & \mapsto & w \phantom{ = \RR(\XX_0)}
}
\]
where $u,v,w$ fulfill the relation $uw=v^2$.
Let $D_1 = \{y=0\}$, it is a divisor of prime order $2$ inside $\Cl(Y)$ and $H^0(Y,D) = \genby{x^2,y}$ as a $\RR(Y;\Pic(Y))$-module. So in the first step of the proof we get $\RR(\YY_1) = \kk[x^2,y]$ and the map
\[
\FctArray{
\Phi_1^* \colon & \RR(\YY_1)=\kk[x^2,y] & \to & \RR(\XX_1)\\
&\phantom{\RR(\YY_1)=} x^2 & \mapsto & (-1)^i \sqrt{u} \\
&\phantom{\RR(\YY_1)=} y & \mapsto & (-1)^j\sqrt{w}
}
\]
Note that $v=(-1)^k \sqrt{u}\sqrt{w}$. Like in the previous example, $i,j\in \ZZ_2$ have to be chosen such that $i+j=k$.  
In a second step, we add the divisor $D_2 = \{x=0\}$ and get $\RR(\YY_2) = \RR(\Ycan) = \RR(Y) = \kk[x,y]$ and the map
\[
\FctArray{
\Phi_2^* \colon & \RR(\YY)= \RR(\YY_2)=\kk[x^2,y] & \to & \RR(\XX_2)=\RR(\XX)\\
& \phantom{\RR(\YY)= \RR(\YY_2)=} x & \mapsto & \sqrt{-1}^{2l+i} \sqrt[4]{u}\\
& \phantom{\RR(\YY)= \RR(\YY_2)=} y & \mapsto & (-1)^j\sqrt{w}
}
\]
with no condition on $l\in \ZZ_2$. 

We just remark that we could have also added the sections of the divisor $D_2$ and its multiples already in a single step and arrive immediately at
\[
\FctArray{
\Phi^* \colon & \kk[x,y] & \to & \RR(\XX)\\
& x & \mapsto & \sqrt{-1}^h \sqrt[4]{u}\\
& y & \mapsto & (-1)^j\sqrt[2]{w}
}
\]
where $h\in \ZZ_4$ has to satisfy $k = 2h +j$.
For the added roots one has to be a bit careful, otherwise graded factoriality will be lost. Let $\mathfrak{q}$ be an h-irreducible factor of multiplicity $a_u$ and $a_v$ in $u$ and $v$, respectively. Denote by $a$ the greatest common divisor of $a_u$ and $a_v$. Then we add no root of $\mathfrak{q}$ if $a=0 \in \ZZ_4$, a square root $\sqrt{\mathfrak{q}}$ if $a=2\in\ZZ_4$ and only a fourth root $\sqrt[4]{\mathfrak{q}}$ if $a$ is invertible in $\ZZ_4$.
\end{example}

\begin{remark}
For the universal property, we require $\RR(\XX)$ to be graded factorial.
Especially, the Cox ring of the fibre product $\XX' = \YY^{\ass} \times_{\YY_0} \Xass$ is not graded factorial, hence will not serve as a Cox lift in general.

Actually, this applies already to \autoref{ex:A2_mod_mu2}. There the fibre product will be given by the $\ZZ_2$-graded algebra
\[
\kk[\sqrt t,u]/u^2,
\]
which essentially comes from adding $u$ as the square root of $0$ (which is the image of $y$), that introduces a zero divisor spoiling graded factoriality.
\end{remark}

\begin{remark}
Given a map $\phi\colon X\to Y$ of Mori dream spaces and its Cox lift $\Phi\colon\XX \to \YY^\ass$, note that $\XX$ will not be smooth in general, even if we assume that $X$ (or $\Xass$) is smooth.
The key point is that when constructing $\XX$, the divisors along which roots are taken may not be simple normal crossing, which is equivalent to $\XX$ being smooth, see \cite[Cor.~1.23]{HM}.
\end{remark}

\section{An application to Mori dream stacks}
\label{sec:application}

\begin{definition}
Let $\XX$ be an algebraic stack separated and of finite type over $\kk$.
We call $\XX$ a \emph{Mori dream stack} (or \emph{\MD-stack} for short) if it fulfills the conditions (1)--(4) of \autoref{def:mdquotientstack}.
\end{definition}

\begin{remark}
The central property of a Mori dream stack here is that its Cox ring is graded factorial and normal. This differs from the definition we gave in the previous article \cite[Def.~2.3]{HM}, where we asked for $\XX$ to be smooth.
\end{remark}

\begin{theorem}
\label{thm:mdquotients-are-roots}
Let $\XX$ be a Mori dream stack whose coarse moduli space $X$ is a Mori dream space with $H^0(\Spec \RR(X) \setminus V,\OO^*)=\kk^*$ and $\Pic(\Spec \RR(X) \setminus V) = 0$.

Then $\XX$ is an \MD-quotient stack if and only if $\XX$ is obtained by root constructions along prime divisors and line bundles.
\end{theorem}

\begin{proof}
For the reverse implication, note that by \autoref{prop:Xass} the canonical \MD-stack of $X$ is a \MD-quotient stack. This property is preserved by root constructions along prime divisors and line bundles due to \autoref{prop:roots-preserve-MD}.

For the forward implication consider the diagram
\[
\xymatrix{
& \XX \ar[d]\\
\Xcan \ar[r]_{\id} \ar[d]  & \Xcan \ar[d]\\
X \ar[r]_{\id} & X
}
\]
where the map $\XX \to \Xcan$ of quotient stacks is given as follows.
The coarse moduli map $\pi\colon \XX \to X$ induces the pullback $\Cl(X) = \Pic(\Xcan) \into \Pic(\XX)$. From this we get a graded inclusion $\RR(X)=\RR(\Xcan) \into \RR(\XX)$, because $H^0(X,\OO(D)) = H^0(\XX,\pi^*\OO(D))$ for any divisor $D$ on $X$. This in turn gives $\XX \to \Xcan$.

Furthermore $\RR(X) = \RR(\Xcan) =  \RR(\XX;\pi^* \Cl(X))$, where the first equality is \autoref{prop:Xass} and the second one follows from the argument above.
Using the proof of \autoref{thm:lift}, we can lift the map $\id\colon \Xcan \to \Xcan$, seen as induced by $\id^* \colon \RR(\XX;\pi^* \Cl(X)) \to \RR(\Xcan)$ to $\mathrm{Id} \colon \XX' \to \XX$, where $\XX'$ is built by roots form $\Xcan$ along prime divisors and line bundles. Moreover, $\XX'$ fits into the top left corner of the diagram above, as $\XX$ itself. By the minimality property, $\XX$ and $\XX'$ are isomorphic.
\end{proof}

\begin{remark}
If $\XX$ is a smooth Mori dream stack, than its coarse moduli space $X$ is a Mori dream space by \cite[Thm.~3.1]{HM}.
Moreover, in that case the canonical stack $\Xass$ of $X$ is a canonical \MD-stack by \cite[Thm.~2.7]{HM}. Consequently, we get an isomorphism $\Cl(X) = \Pic(\Xass)$, so  the technical assumptions $H^0(\Spec \RR(X) \setminus V,\OO^*)=\kk^*$ and $\Pic(\Spec \RR(X) \setminus V) = 0$ are fulfilled as well, by \autoref{rem:KKV_sequence}.
\end{remark}

\begin{remark}
In the analogous statement of our previous article \cite[Main Theorem (2)]{HM}, we relate smooth \MD-quotient stacks with stacks obtained by roots along simple normal crossing divisors (and line bundles).
\autoref{thm:mdquotients-are-roots} generalises this statement in the following sense: if $\XX$ is a smooth \MD-stack (as in \cite[Def.~2.3]{HM}) with a simple normal crossing ramification divisor, then its Cox ring is graded factorial and normal.

Let $\XX$ be such a stack.  By \cite[Thm.~3.1]{HM}, its coarse moduli space $X$ is a  Mori dream space, thus $\RR(X)$ is factorial and normal. By \cite[Prop.~2.7]{HM}, the smooth canonical stack $\Xcan$ is a \MD-stack with $\RR(\Xcan)=\RR(X)$.

Now, if $\XX$ is a smooth orbifold, by \cite[Prop.~3.2]{HM} it is obtained from $\Xcan$ by roots along simple normal crossing divisors. Then its Cox ring is graded factorial and normal by \autoref{lem:algebra-map} and \autoref{lem:cover-normal}.

If $\XX$ is not an orbifold, then it is a gerbe over its rigidification $\Xrig$, which is a smooth orbifold. Since gerbe constructions leave the Cox ring unchanged by \cite[Thm.~2.10]{HM}, we can conclude that also in this case $\RR(\XX)$ is graded factorial and normal.
\end{remark}

\input{biblio-quotient-maps}


\end{document}

%% file: header-quotient.tex
\usepackage[T1]{fontenc}
\usepackage{amsmath,amssymb,amsthm,amsfonts,bbm}
\usepackage{color}
\usepackage[all]{xy}
\SelectTips{cm}{}
\usepackage{tikz}
\usepackage{hyperref}
\usepackage{stmaryrd} 
\usepackage{enumitem}
\usepackage{multicol}
\usepackage{graphicx}

\setitemize[0]{leftmargin=*}
\setenumerate[0]{leftmargin=*}
\let\OLDenumerate\enumerate
\renewcommand\enumerate{\OLDenumerate\addtolength{\itemsep}{0.5ex}}

\definecolor{rosso}{RGB}{162,0,0}
\definecolor{verde}{RGB}{0,100,0}
\definecolor{blu}{RGB}{0,0,162}
\definecolor{bla}{RGB}{100,0,100}

\newcommand{\domark}{%
  \vbox to 0pt{
    \kern-\dp\strutbox
    \smash{\llap{*\kern1em}}
    \vss
  }%
}

\newcommand{\MD}{\textsf{MD}}

\newcommand{\QQ}{\mathbb{Q}}
\newcommand{\ZZ}{\mathbb{Z}}
\newcommand{\NN}{\mathbb{N}}
\renewcommand{\AA}{\mathbb{A}} 

\newcommand{\XX}{\mathcal{X}} 
\newcommand{\YY}{\mathcal{Y}} 
\newcommand{\can}{{\mathop{can}}}
\newcommand{\ass}{{\mathop{can}}} 
\newcommand{\rig}{{\mathop{rig}}}

\newcommand{\Xcan}{\mathcal{X}^{\can}}
\newcommand{\Xass}{{\mathcal{X}^{\ass}}}
\newcommand{\Xrig}{\mathcal{X}^{\rig}}

\newcommand{\Ycan}{\mathcal{Y}^{\can}}

\newcommand{\OO}{\mathcal{O}} 
\newcommand{\LL}{\mathcal{L}} 
\newcommand{\KK}{\mathcal{K}} 
\newcommand{\RR}{\mathcal{R}} 
\renewcommand{\SS}{\mathcal{S}} 
\newcommand{\II}{\mathcal{I}}

\newcommand{\quot}[2]{\left. \raisebox{0.5ex}{$\displaystyle #1$} \middle/ \raisebox{-0.5ex}{$\displaystyle #2$} \right.}

\newcommand{\stackquot}[2]{\left[ \raisebox{0.5ex}{$\displaystyle #1$} \middle/ \raisebox{-0.5ex}{$\displaystyle #2$} \right]}

\newcommand{\kk}{\mathbbm{k}} 
\newcommand{\pt}{\mathsf{pt}} 

\newcommand{\BB}{\mathcal{B}} 

\newcommand{\qq}{\mathfrak{q}} 

\DeclareMathOperator{\im}{im}

\DeclareMathOperator{\Hom}{Hom}

\let\mod\relax
\DeclareMathOperator{\mod}{mod}

\DeclareMathOperator{\id}{id}
\let\div\relax
\DeclareMathOperator{\div}{div}

\DeclareMathOperator{\CDiv}{CDiv}
\DeclareMathOperator{\Pic}{Pic}

\DeclareMathOperator{\WDiv}{WDiv}
\DeclareMathOperator{\Cl}{Cl}
\DeclareMathOperator{\Spec}{Spec}
\DeclareMathOperator{\supp}{supp}

\DeclareMathOperator{\sm}{sm} 
\DeclareMathOperator{\sing}{sing} 
\DeclareMathOperator{\irr}{irr} 
\DeclareMathOperator{\Jirr}{J_{\irr}} 

\newcommand{\genby}[1]{{\langle #1 \rangle}} 
 
\newcommand{\coloneqq}{\mathrel{\mathop:}=}
\newcommand{\eqqcolon}{=\mathrel{\mathop:}}
\newcommand{\xto}{\xrightarrow}
\newcommand{\into}{\hookrightarrow}
\newcommand{\onto}{\twoheadrightarrow}

\newcommand{\FctArray}[1]{\begin{array}{r*{5}{@{\:}c}c} #1 \end{array}} 

\newtheorem{theorem}{Theorem}[section]
\newtheorem*{maintheorem}{Main Theorem}
\newtheorem*{theorem*}{Theorem}
\newtheorem*{lemma*}{Lemma}

\theoremstyle{definition}
\newtheorem*{example*}{Example}

\newcommand{\step}[1]{\par\medskip\par\noindent\textit{#1.}} 

\newcommand{\bib}[6]{{\bibitem{#2} #3, {\emph{#4},} #5#6.}}
\newcommand{\arXiv}[1]{{\href{http://arxiv.org/abs/#1}{\texttt{arXiv:#1}}}}

\hypersetup{
    colorlinks,
    linkcolor={red!50!black},
    citecolor={blue!50!black},
    urlcolor={blue!80!black}
}
\usepackage{aliascnt} 

  \newaliascnt{proposition}{theorem}
  \newtheorem{proposition}[proposition]{Proposition}
  \aliascntresetthe{proposition}

  \newaliascnt{lemma}{theorem}
  \newtheorem{lemma}[lemma]{Lemma}
  \aliascntresetthe{lemma}

  \newaliascnt{corollary}{theorem}
  
  \aliascntresetthe{corollary}

\theoremstyle{definition}

  \newaliascnt{definition}{theorem}
  \newtheorem{definition}[definition]{Definition}
  \aliascntresetthe{definition}

  \newaliascnt{remark}{theorem}
  \newtheorem{remark}[remark]{Remark}
  \aliascntresetthe{remark}

  \newaliascnt{example}{theorem}
  \newtheorem{example}[example]{Example}
  \aliascntresetthe{example}


%% file: biblio-quotient-maps.tex
\addtocontents{toc}{\protect\setcounter{tocdepth}{-1}}